\documentclass[a4paper,12pt]{amsart}

\usepackage{amsfonts}
\usepackage{amsmath}
\usepackage{amssymb}

\usepackage{mathrsfs}
\usepackage{hyperref}

\theoremstyle{plain}
\newtheorem{thm}{Theorem}[section]
\newtheorem{prop}[thm]{Proposition}
\newtheorem{cor}[thm]{Corollary}
\newtheorem{lem}[thm]{Lemma}

\theoremstyle{definition}

\newtheorem{rem}[thm]{Remark}

\newcommand{\be}{\begin{equation}}
\newcommand{\ee}{\end{equation}}

\def\cal{\mathcal}

\def\R{{\mathbb R}}
\def\D{{\mathbb D}}

\def\N{{\mathbb N}}

\def\C{{\mathbb C}}
\def\S3{{{\mathbb S}^3}}
\def\SU2{{{\rm SU}(2)}}
\def\Rn{{{\mathbb R}^n}}
\def\Tn{{{\mathbb T}^n}}
\def\Zn{{{\mathbb Z}^n}}

\def\Gh{{\widehat{G}}}
\def\Lap{{\mathcal L}}

\def\p#1{{\left({#1}\right)}}

\def\jp#1{{\left\langle{#1}\right\rangle}}
\def\n#1{{\left\|{#1}\right\|}}
\def\abs#1{{\left|{#1}\right|}}

\def\Rep{{{\rm Rep}}}
\def\Op{{{\rm Op}}}
\def\Repa{{{\rm Re}}}

\def\Dcal{{\mathcal D}}

\def\Hcal{{\mathcal H}}

\def\drm{{\ {\rm d}}}
\def\irm{{\rm i}}
\def\erm{{\ {\rm e}}}
\def\va{\varphi}

\def\Ccl{{\mathscr S}}

\def\G{{G}}

\DeclareMathOperator{\Tr}{Tr}
\def\diff{\mathrm{diff}}
\def\Diff{\mathrm{Diff}}
\DeclareMathOperator{\rank}{rank}
\DeclareMathOperator{\singsupp}{sing\,supp}
\def\d{\mathrm d}

\begin{document}

\title[Sharp G{\aa}rding inequality on compact Lie groups]
{Sharp G{\aa}rding inequality on compact Lie groups}

\author[Michael Ruzhansky]{Michael Ruzhansky}
\address{
  Michael Ruzhansky:
  \endgraf
  Department of Mathematics
  \endgraf
  Imperial College London
  \endgraf
  180 Queen's Gate, London SW7 2AZ 
  \endgraf
  United Kingdom
  \endgraf
  {\it E-mail address} {\rm m.ruzhansky@imperial.ac.uk}
  }

\author[Ville Turunen]{Ville Turunen}
\address{
  Ville Turunen:
  \endgraf
   Aalto University
  \endgraf
  Institute of Mathematics
  \endgraf
   P.O. Box 1100
  \endgraf
   FI-00076 AALTO
  \endgraf
  Finland
  \endgraf
  {\it E-mail address} {\rm ville.turunen@hut.fi}
  }

\thanks{The first
 author was supported in part by the EPSRC
 Leadership Fellowship EP/G007233/1.}
\date{\today}

\subjclass{Primary 35S05; Secondary 22E30}
\keywords{Pseudo-differential operators, compact Lie groups,
microlocal analysis, G{\aa}rding inequality}

\begin{abstract}
We establish the sharp G{\aa}rding inequality on compact Lie groups.
The positivity condition is expressed in the non-commutative
phase space in terms of the full symbol, which is defined using the
representations of the group. Applications are given to
the $L^2$ and Sobolev boundedness of pseudo-differential operators.
\end{abstract}

\maketitle

\section{Introduction}

The sharp G{\aa}rding inequality on $\Rn$ is one of the most
important tools of the microlocal analysis with numerous
applications in the theory of partial differential equations.
Improving on the original G{\aa}rding inequality in \cite{Garding},
H\"ormander \cite{Hormander} showed that if 
$p\in S^m_{1,0}(\Rn)$ and $p(x,\xi)\geq 0$, then
\begin{equation}\label{EQ:Ga-sharp1}
\Repa (p(x,D)u,u)_{L^2}\geq -C\|u\|^2_{H^{(m-1)/2}}
\end{equation} 
holds for all $u\in C_0^\infty(\Bbb R^n)$. The scalar case
was also later extended to matrix-valued operators
by Lax and Nirenberg \cite{LaxNirenberg}, 
Friedrichs \cite{Friedrichs} and Vaillancourt
\cite{Va}. Further improvements on the lower bound
in the scalar case were also obtained by
Beals and Fefferman \cite{BF} and Fefferman and Phong
\cite{FP}. 

Notably, the sharp G{\aa}rding inequality \eqref{EQ:Ga-sharp1}
requires the condition $p(x,\xi)\geq 0$ imposed on the
full symbol. This is different from the original
G{\aa}rding inequality for elliptic operators which
can be readily extended to manifolds. 
The main difficulty in obtaining \eqref{EQ:Ga-sharp1} in the
setting of manifolds is that the full symbol of an operator
can not be invariantly defined via its localisations.
While the standard localisation 
approach still yields the principal symbol
and thus the standard G{\aa}rding inequality, it can not
be extended to produce an improvement of the type in
\eqref{EQ:Ga-sharp1}. Nevertheless, for
pseudo-differential operators
$P\in\Psi^2(M)$ on a compact
manifold $M$, under certain geometric 
restrictions on the characteristic variety of the principal 
symbol $p_2\geq 0$ and certain hypothesis on $p_1$, 
Melin \cite{Melin} and H\"ormander \cite{Horm2} obtained
a lower bound known as the H\"ormander--Melin
inequality. See also Taylor \cite{Tay-book}.

The aim of the present paper is to
establish the lower bound \eqref{EQ:Ga-sharp1}
on any compact Lie group $G$,
with the statement given in Theorem \ref{THM:Garding}.
On compact Lie groups, the non-commutative analogue of the phase space
is $G\times\Gh$, where $\Gh$ is the unitary dual of $G$.
We use a global quantization of operators
on $G$ consistently developed by the 
authors in \cite{RT-groups} and \cite{RT-book}.
For a continuous linear operator $A:C^\infty(G)\to\Dcal'(G)$
it produces a full matrix-valued symbol $\sigma_A(x,\xi)$
defined for $(x,[\xi])\in G\times\Gh$.
Thus, in Theorem \ref{THM:Garding} we will show 
the lower bound \eqref{EQ:Ga-sharp1} under the assumption
that the full symbol satisfies $\sigma_A\geq 0$,
i.e. when the matrices $\sigma_A(x,\xi)$ are positive
for all $(x,[\xi])\in G\times\Gh$.
In general, if a full symbol is positive in the phase space,
the corresponding pseudo-differential operator
does not have to be positive in the operator sense. However, 
it still has lower bounds like the one in \eqref{EQ:Ga-sharp1}.
An important example is the group
$\SU2\cong\mathbb S^3$, with the group operation (matrix product) in
$\SU2$ corresponding to the quaternionic product in
$\mathbb S^3$. Details of the global quantization
have been worked out in 
\cite{RT-book, RT-groups}.

We note that the standard
G{\aa}rding inequality on compact Lie groups was derived in
\cite{BGJR} using Langlands' results for semigroups
on Lie groups \cite{Lang}, but no quantization yielding
full symbols is required in this case because of the
ellipticity assumed on the operator. The global quantization
used in \cite{RT-groups} and \cite{RT-book} will be briefly
reviewed in Section \ref{SEC:Ga-prelim}. We note that it is
different from the one considered by Taylor \cite{Tay-noncomm}
because we work directly on the group without referring to the
exponential mapping and the symbol classes on the Lie algebra.

We note that one of the assumptions for the H\"ormander--Melin
inequality to hold is the vanishing of the principal symbol
$p_2\geq 0$ on the set $\{p_2=0\}$ to exactly second order.
Thus, for example, it does not apply to operators of the
form $-\partial_X^2$ plus lower order terms, where 
$\partial_X$ is 
the derivative with respect to a vector field $X$,
unless $\dim G=1$. For higher order operators, again,
the operator $\partial_X^4-\Lap_G$, with the bi-invariant Laplace
operator $\Lap_G$,
gives an example when the H\"ormander--Melin inequality
does not work while the full matrix--valued symbol is 
positive definite, so that Theorem \ref{THM:Garding} applies.
The relaxation of the transversal ellipticity has been analysed
recently by Mughetti, Parenti and Parmeggiani, and we refer
to \cite{MPP} for further details on this subject.

A usual proof of \eqref{EQ:Ga-sharp1} in $\Rn$ relies on
the Friedrichs symmetrisation of an operator done in the
frequency variables (\cite{Hormander}, \cite{LaxNirenberg},
\cite{Friedrichs}, \cite{Kumano-go}, \cite{Tay-book}).
This does not readily work in the setting of Lie groups
because the unitary dual $\Gh$ forms only a lattice
which does not behave well enough for this type of arguments.
Thereby
our construction uses mollification in 
$x$-space, more resembling those used by
Calder\'on \cite{Ca} or Nagase \cite{Na} for the
proof of the sharp G{\aa}rding inequality in $\Rn$.
Other proofs, e.g. using the anti-Wick quantisation, are also
available on $\Rn$, see \cite{Ler} and references therein.
We would also like to point out that the proof of
Nagase \cite{Na} can be extended to prove \eqref{EQ:Ga-sharp1}
on the torus $\Tn$ under the assumption that the toroidal
symbol $p(x,\xi)$ of the operator $P\in\Psi^m(\Tn)$
satisfies $p(x,\xi)\geq 0$ for all $x\in\Tn$ and $\xi\in\Zn$.
The toroidal quantization necessary for this proof was
developed by the authors in \cite{RT-torus} but we will not
give such a proof here because such result is now included
as a special
case of Theorem \ref{THM:Garding} which covers the
non-commutative groups as well. The system as well as
$(\rho,\delta)$ versions of the sharp G{\aa}rding inequality
will appear elsewhere.

The proof of Theorem \ref{THM:Garding} consists of
approximating the operator $A$ with non-negative symbol $\sigma_A$
by a positive operator $P$.
Although this approximation has
a symbol of type $(1,1/2)$ and not of type $(1,0)$, it is
enough to prove Theorem \ref{THM:Garding} due to additional
cancellations in the error terms, ensured by the construction.
We note that working with symbol classes of type
$(1,1/2)$ is a genuine global feature of the proof and of
our construction because
the operators of such type can not be defined in local
coordinates.

As usual, for a compact Lie group $G$ we denote by
$\Psi^m(G)$ the H\"ormander pseudo-differential 
operators on $G$, i.e. the class of operators which
in all local coordinate charts give operators in
$\Psi^m(\Rn)$. Operators in $\Psi^m(\Rn)$ are characterised
by the symbols satisfying
$$
|\partial_\xi^\alpha\partial_x^\beta a(x,\xi)|\leq
C(1+|\xi|)^{m-|\alpha|}
$$
for all multi-indices $\alpha, \beta$ and all
$x,\xi\in\Rn$. An operator in $\Psi^m(G)$ is called elliptic
if all of its localisations are locally elliptic.
Here and in the sequel we use the standard notation for the
multi-indices $\alpha=(\alpha_1,\ldots,\alpha_\mu)\in\N_0^\mu$,
where $\mu$ may vary throughout the paper depending on the
context. 

The paper is organised as follows. In Section 
\ref{SEC:GA-sharp-ineq} we introduce the full matrix-valued
symbols and state the sharp G{\aa}rding inequality in Theorem
\ref{THM:Garding}. We apply it
in Corollaries \ref{COR:Ga-Taylorthm} and \ref{COR:Ga-Kgothm}
to the $L^2$ boundedness of pseudo-differential operators.
In Section \ref{SEC:Ga-prelim} we collect
facts necessary for the proof, and develop
an expansion of amplitudes of type $(\rho,\delta)$
required for our analysis. In Section \ref{SEC:Ga-proofs}
we approximate operators with positive symbols
by positive operators and derive the error estimates.

The authors would like to 
thank Jens Wirth for discussions and a referee for useful remarks.

\section{Sharp G{\aa}rding inequality}
\label{SEC:GA-sharp-ineq}

Let $G$ be a compact Lie group of dimension $n$
with the neutral element $e$. Its Lie algebra will be denoted by
$\mathfrak g$.
We now fix the necessary notation. Let $\Gh$
denote the unitary dual of $G$, i.e. 
set of all equivalence classes of
(continuous) irreducible unitary
representations of $G$ and let $\Rep(G)$ be the set of
all such representations of $G$. 

For $f\in C^\infty(G)$ and
$\xi\in\Rep(G)$, let
$$\widehat{f}(\xi)=\int_G f(x)\ \xi(x)^*\drm x$$ 
be the
(global) Fourier transform of $f$,
where integration is with respect to the normalised 
Haar measure on $G$. 
For an irreducible unitary representation 
$\xi:G\to{\cal U}({\cal H}_\xi)$ we have the linear operator
$
  \widehat{f}(\xi):{\cal H}_\xi\to{\cal H}_\xi.
$
Denote by $\dim(\xi)$ the
dimension of $\xi$, $\dim(\xi)=\dim {\cal H}_\xi$.
If $\xi$ is a matrix representation,
we have $\widehat{f}(\xi)\in\C^{\dim(\xi)\times\dim(\xi)}.$
Since $G$ is compact, $\Gh$ is discrete and all of its
elements are finite dimensional.
Consequently, by the Peter--Weyl theorem we have
the Fourier inversion formula
$$
  f(x)=\sum_{[\xi]\in\Gh}\dim(\xi)\Tr\left(\xi(x)\ \widehat{f}(\xi)\right).
$$
The Parseval identity takes the form
$$\|f\|^2_{L^2(G)}=\sum_{[\xi]\in\Gh}\dim(\xi)
\|\widehat{f}(\xi)\|_{HS}^2,$$ 
where
$\|\widehat{f}(\xi)\|_{HS}^2=\Tr(\widehat{f}(\xi)
\widehat{f}(\xi)^*)$,
which gives the
norm on $\ell^2(\Gh)$. For a linear continuous
operator from $C^\infty(G)$ to $\Dcal'(G)$ we introduce its
{\em full matrix-valued symbol} $\sigma_A(x,\xi)\in\C^{\dim(\xi)
\times\dim(\xi)}$ by
$$
\sigma_A(x,\xi)=\xi(x)^* (A\xi)(x).
$$
Then it was shown in \cite{RT-groups} and \cite{RT-book} that
\begin{equation}\label{EQ:Ga-A-op-def}
  Af(x)
  =\sum_{[\xi]\in\Gh}\dim(\xi)\Tr\p{\xi(x)\ \sigma_A(x,\xi)\ \widehat{f}(\xi)}
\end{equation} 
holds in the sense of distributions, and the sum is independent
of the choice of a representation $\xi$ from each class
$[\xi]\in\Gh$. Moreover, we have 
$$\sigma_A(x,\xi)=\int_G R_A(x,y)\ \xi(y)^*\drm y$$ 
in the sense of distributions,
where
$R_A$ is the right-convolution kernel of $A$:
$$
  Af(x)=\int_G K(x,y)\ f(y)\drm y=\int_G f(y)\ 
  R_A(x,y^{-1}x)\drm y.
$$
Symbols $\sigma_A$ can be viewed as mappings on $G\times\Gh$:
the symbol of a continuous linear operator
$A:C^\infty(G)\to C^\infty(G)$ is a mapping
\begin{equation*}
  \sigma_A:G\times\Rep(G)
  \to\bigcup_{\xi\in\Rep(G)}{\rm End}
  ({\Hcal}_\xi),
\end{equation*}
where $\sigma_A(x,\xi):{\Hcal}_\xi\to{\Hcal}_\xi$ is linear
for every $x\in G$ and $\xi\in{\rm Rep}(G)$,
see \cite[Rem.~10.4.9]{RT-book}, and
${\rm End}({\Hcal}_\xi)$ is the space of all linear mappings
from ${\Hcal}_\xi$ to ${\Hcal}_\xi$.
If $\eta\in[\xi]$, i.e. there is an intertwining isomorphism
$U:{\cal H}_\eta\to{\cal H}_\xi$ such that
$\eta(x)=U^{-1}\xi(x)U$, then
$\sigma_A(x,\eta)=U^{-1}\sigma_A(x,\xi)U$.
In this sense we may think that the symbol $\sigma_A$
is defined on $G\times\Gh$ instead of $G\times\Rep(G)$.
For further details of these constructions and their
properties we refer to \cite{RT-book}.

A (possibly unbounded) linear operator $P$ 
on a Hilbert space $\cal H$
is called {\it positive} if $\langle Pv,v\rangle_{\cal H}\geq 0$
for every $v\in V$ for a dense subset $V\subset{\cal H}$.
A matrix $P\in\Bbb C^{n\times n}$ is called positive
if the natural corresponding linear operator $\Bbb C^n\to\Bbb C^n$
is positive, where $\Bbb C^n$ has the standard inner product.

A matrix pseudo-differential symbol $\sigma_A$
is called {\em positive} if
the matrix
$\sigma_A(x,\xi)\in{\Bbb C}^{{\rm dim}(\xi)\times{\rm dim}(\xi)}$
is positive for every $x\in G$ and $\xi\in\Rep(G)$.
In this case we write $\sigma(x,\xi)\geq 0$.
We note that for each $\xi\in\Rep(G)$, the condition
$\sigma_A(x,\xi)\geq 0$ implies $\sigma_A(x,\eta)\geq 0$
for all $\eta\in [\xi]$.
We can also note that this symbol positivity 
does not change if we move from left symbols to right symbols:
\begin{eqnarray*}
  \sigma_A(x,\xi) := \xi^\ast(x)\ \left( A\xi\right)(x)
  = \xi(x)^\ast\ \rho_A(x,\xi)\ \xi(x), \\
  \rho_A(x,\xi) := \left( A\xi\right)(x)\ \xi^\ast(x)
  = \xi(x)\ \sigma_A(x,\xi)\ \xi(x)^\ast;
\end{eqnarray*}
that is, $\sigma_A$ is positive if and only if $\rho_A$ is positive.
Moreover, this positivity concept is natural in the sense that
a left- or right-invariant operator is positive
if and only if its symbol is positive,
as it can be seen from the equalities
\begin{eqnarray}\label{EQ:Ga-multip-pos1}
  \langle a\ast f,f\rangle_{L^2(G)}
  & = & \sum_{[\xi]\in\widehat{G}} {\rm dim}(\xi)
  \ {\rm Tr}\left(\widehat{f}(\xi)\
  \widehat{a}(\xi)\ \widehat{f}(\xi)^\ast
    \right), \\
  \label{EQ:Ga-multip-pos2}
  \langle f\ast a,f\rangle_{L^2(G)}
  & = & \sum_{[\xi]\in\widehat{G}} {\rm dim}(\xi)
  \ {\rm Tr}\left(\widehat{f}(\xi)^\ast\ \widehat{a}(\xi)
  \ \widehat{f}(\xi)  \right),
\end{eqnarray}
which can be shown by a simple calculation which we give in
Proposition~\ref{PROP:Ga-multip-pos}. At the same time,
the operator $M_f$ of multiplication by a smooth function 
$f\in C^\infty(G)$ is
positive if and only if
the function satisfies $f(x)\geq 0$ for every $x\in G$.
The symbol of such multiplication operator is
$\sigma_{M_f}(x,\xi)=f(x) I_{\dim(\xi)}$,
so that this means the positivity of the matrix symbol again.

Now we can formulate the main result of this paper:

\begin{thm}\label{THM:Garding}
Let $A\in\Psi^m(G)$ be such that its full matrix symbol
$\sigma_A$ satisfies
$\sigma_A(x,\xi)\geq 0$ for all $(x,[\xi])\in G\times\Gh$.
Then there exists $C<\infty$ such that
$$
  \Repa (Au,u)_{L^2(G)} \geq -C \|u\|_{H^{(m-1)/2}(G)}^2
$$
for every $u\in C^\infty(G)$.
\end{thm} 

As a corollary of Theorem~\ref{THM:Garding} we obtain the
following statement on compact Lie groups, analogous to the corresponding
result on $\Rn$, which is often necessary in the proofs of
pseudo-differential inequalities (see e.g.
Theorem 3.1 in \cite{Tay-book}).

\begin{cor}\label{COR:Ga-Taylorthm}
Let $A\in\Psi^1(G)$ be such that its matrix symbol
$\sigma_A$ satisfies
$$
  \left\|\sigma_A(x,\xi)\right\|_{op}\leq C
$$
for all $(x,[\xi])\in G\times\Gh$. Then $A$ is
bounded from $L^2(G)$ to $L^2(G)$.
\end{cor} 
Here $\|\cdot\|_{op}$ denotes the 
$\ell^2\to\ell^2$
operator norm of the 
linear finite dimensional mapping (matrix multiplication by)
$\sigma_A(x,\xi)$, i.e.
$$\|\sigma_A(x,\xi)\|_{op}=\sup\{\|\sigma_A(x,\xi)v\|_{\ell^2}:
v\in\mathbb C^{\dim(\xi)}, \|v\|_{\ell^2}\leq 1\}.$$

The weights for measuring the orders of symbols are expressed
in terms of the eigenvalues of the bi-invariant Laplacian $\mathcal L_G$.
Matrix elements of every representation class $[\xi]\in\Gh$ 
span an eigenspace of the bi-invariant Laplace--Beltrami
operator $\mathcal L_G$ on $\G$ with the
corresponding eigenvalue $-\lambda_\xi^2$. Based on these 
eigenvalues we define 
$$\langle\xi\rangle = (1+\lambda_\xi^2)^{1/2}.$$ 
For further details and properties of these
constructions we refer to \cite{RT-book}. 
In particular, for the usual Sobolev spaces, we have
$f\in H^s(G)$ if and only if
$\jp{\xi}^s\widehat{f}(\xi)\in\ell^2(\Gh)$.
To fix the norm on $H^s(G)$ for the following statement, we
can then set
$$\|f\|_{H^s(G)}:=\left(\sum_{[\xi]\in\Gh}\dim(\xi)\jp{\xi}^{2s}
\Tr(\widehat{f}(\xi)^\ast\widehat{f}(\xi))\right)^{1/2},$$
and we can write this also as $\|\jp{\xi}^s\widehat{f}(\xi)\|_{\ell^2(\Gh)}.$
Also, we note that by \cite[Lemma~10.9.1]{RT-book}
(or by Theorem \ref{thm:main} below), if $A\in\Psi^m(G)$, then
there is a constant $0<M<\infty$ such that
$\|\sigma_A(x,\xi)\|_{op}\leq M\jp{\xi}^m$ holds for all
$x\in G$ and $[\xi]\in\Gh$.

As another corollary of Theorem~\ref{THM:Garding} we can get
a norm-estimate for pseudo-differential operators
on compact Lie groups:

\begin{cor}\label{COR:Ga-Kgothm}
Let $A\in\Psi^m(G)$ and let 
$$M=\sup_{(x,[\xi])\in G\times\Gh} 
\p{\jp{\xi}^{-m}\|\sigma_A(x,\xi)\|_{op}}.$$ 
Then 
for every $s\in\R$ there
exists a constant $C>0$ such that
$$
  \|Au\|_{H^s(G)}^2\leq M^2\|u\|_{H^{s+m}(G)}^2+
  C\|u\|_{H^{s+m-1/2}(G)}^2
$$
for all $u\in C^\infty(G)$.
\end{cor} 

\section{Preliminary constructions}\label{SEC:Ga-prelim}

In this section we collect and develop several ideas which will be
used in the proof of Theorem~\ref{THM:Garding}. These
include characterisations of the class $\Psi^m(G)$,
the Leibniz formula,
the amplitude operators on $G$, 
and some properties of even and odd functions.

\subsection{On symbols and operators}

First we collect several facts and definitions required
for our proof.
We now introduce the notation for the symbol classes
on the group $G$ and give a characterisation of classes
$\Psi^m(G)$ in terms of the matrix-valued symbols.
In this, we follow the notation of \cite{RTW}.

We say that $Q_\xi$ is a {\em difference operator} 
of order $k$ if it is given by
$$
   Q_\xi \widehat f(\xi) = \widehat{q_Q f}(\xi),
$$
for a function $q=q_Q\in C^\infty(\G)$ vanishing of order $k$ at
the identity $e\in\G$, i.e., $(P_x q_Q)(e)=0$ for 
all left-invariant differential operators 
$P_x\in\Diff^{k-1}(\G)$ of order $k-1$. 
We denote the set of all difference operators of order $k$ as 
$\diff^k(\Gh)$.

A collection of $\mu\ge n$ first order difference 
operators 
$\triangle_1, \ldots, \triangle_\mu\in\diff^1(\widehat\G)$
is called {\em admissible}, if the corresponding 
functions $q_1, \ldots, q_\mu\in C^\infty(\G)$ satisfy 
$q_j(e)=0$,
$\d q_j(e)\ne 0$ for all $j=1,\ldots,\mu$, and if
$\rank(\d q_1(e),\ldots,\d q_\mu(e))=n$. 
An admissible collection is called {\em strongly admissible} if 
$\bigcap_{j=1}^\mu \{ x\in\G : q_j(x)=0\} = \{e\}$.

For a given admissible selection of difference operators 
on a compact Lie group $G$ we use multi-index notation 
$\triangle_\xi^\alpha = \triangle_1^{\alpha_1}\cdots 
\triangle_\mu^{\alpha_\mu}$ and 
$q^\alpha(x) = q_1(x)^{\alpha_1}\cdots q_\mu(x)^{\alpha_\mu}$ . 
Furthermore, there exist corresponding differential operators 
$\partial_x^{(\alpha)}\in \Diff^{|\alpha|}(\G)$ 
such that Taylor's formula
\begin{equation}\label{EQ:RTW-Taylor-exp}
   f(x) = \sum_{|\alpha|\le N-1} \frac1{\alpha!} \, 
   q^\alpha(x) \, \partial_x^{(\alpha)} f(e) + 
   \mathcal O( {\rm dist}(x,e)^N )
\end{equation} 
holds true for any smooth function $f\in C^\infty(\G)$ 
and with ${\rm dist}(x,e)$ the geodesic distance from $x$ to 
the identity element $e$. An explicit construction of
operators $\partial_x^{(\alpha)}$ in terms of $q^\alpha(x)$
can be found in \cite[Section 10.6]{RT-book}.
In addition to these differential operators
$\partial_x^{(\alpha)}\in \Diff^{|\alpha|}(\G)$ we introduce
operators $\partial_x^{\alpha}$ as follows. Let
$\{\partial_{x_j}\}_{j=1}^n\subset \Diff^{1}(\G)$ be a 
collection of left-invariant first order differential
operators corresponding to some linearly independent family
of the left-invariant vector fields on $G$. We denote
$\partial_x^{\alpha}=\partial_{x_1}^{\alpha_1}\cdots
\partial_{x_n}^{\alpha_n}$. We note that in most estimates
we can freely replace operators $\partial_x^{(\alpha)}$
by $\partial_x^{\alpha}$ and in the other way around since
they can be expressed in terms of each other.
For further details and properties of the introduced
constructions we refer to \cite{RT-book}. 

We now record the characterisation of H\"ormander's classes
as it appeared in \cite{RTW}:

\begin{thm}\label{thm:main}
Let $A$ be a linear continuous operator from 
$C^\infty(\G)$ to $\mathcal D'(\G)$, and let $m\in\Bbb R$.
Then the following statements are equivalent:
\begin{enumerate}
\item[(A)] 
$A\in\Psi^m(\G)$.
\item[(B)] 
For every left-invariant differential operator 
$P_x\in\Diff^k(\G)$ of order $k$ and every difference operator 
$Q_\xi\in\diff^\ell(\widehat\G)$ of order $\ell$ 
there is the symbol estimate
$$
  \| Q_\xi P_x  \sigma_A(x,\xi) \|_{op} \le C_{Q_\xi P_x} 
  \langle\xi\rangle^{m-\ell}.
$$
\item[(C)] 
For an admissible collection 
$\triangle_1,\ldots,\triangle_\mu \in\diff^1(\widehat\G)$  we have
$$
  \| \triangle_\xi^\alpha \partial_x^{\beta} 
  \sigma_A(x,\xi)\|_{op} \le C_{\alpha\beta} 
  \langle\xi\rangle^{m-|\alpha|}
$$
for all multi-indices $\alpha,\beta$.
Moreover, $\singsupp R_A(x,\cdot) \subseteq \{e\}$.
\item[(D)] 
For a strongly admissible collection 
$\triangle_1,\ldots,\triangle_\mu \in\diff^1(\widehat\G)$ we have
$$
  \| \triangle_\xi^\alpha \partial_x^{\beta} 
  \sigma_A(x,\xi)\|_{op} \le C_{\alpha\beta} 
  \langle\xi\rangle^{m-|\alpha|}
$$
for all multi-indices $\alpha,\beta$.
\end{enumerate}
\end{thm}
The set of symbols $\sigma_A$ satisfying either of
conditions $(B)$--$(D)$ will be denoted by
$\mathscr S^{m}_{1,0}(G)=\mathscr S^{m}(G)$. 
We note that if conditions (C) or (D) hold for one
admissible (strongly admissible, resp.) collection
of first order difference operators, they automatically
hold for all admissible (strongly admissible, resp.)
collections.

For the
purposes of this paper, we will also need larger classes
of symbols which we now introduce. We will say that a
matrix-valued symbol $\sigma_A(x,\xi)$ belongs to
$\Ccl^m_{\rho,\delta}(G)$ if it is smooth in $x$ and if
for a strongly admissible collection 
$\triangle_1,\ldots,\triangle_\mu \in\diff^1(\widehat\G)$  we have
\begin{equation}\label{EQ:Ga-rhodel-classes}
  \| \triangle_\xi^\alpha \partial_x^{\beta} 
  \sigma_A(x,\xi)\|_{op} \le C_{\alpha\beta} 
  \langle\xi\rangle^{m-\rho|\alpha|+\delta|\beta|}
\end{equation} 
for all multi-indices $\alpha,\beta$,
uniformly in $x\in G$ and $\xi\in\Rep(G)$.
\begin{rem}\label{REM:Ga-rhodel-classes}
As it was pointed out in \cite{RTW},
in Theorem \ref{thm:main}
we still have
the equivalence of conditions $(B)$, $(C)$, $(D)$, also
if we replace symbolic inequalities in 
Theorem \ref{thm:main} by inequalities of the form
\eqref{EQ:Ga-rhodel-classes}.
Also in this setting, if conditions (C) or (D) hold for one
admissible (strongly admissible, resp.) collection
of first order difference operators, they automatically
hold for all admissible (strongly admissible, resp.)
collections.
\end{rem} 

We will also write $a\in \Ccl^m_{\rho,\delta \#}(G)$ if
for every multi-index $\beta$ and for every $x_0\in G$ we have
$\partial_x^\beta a(x_0,\cdot)\in \Ccl^{m+\delta|\beta|}_{\rho\#}(G)$,
where for a multiplier $b=b(\xi)$ we write
$b\in \Ccl^{\mu}_{\rho\#}(G)$ if for every
multi-index $\alpha$ there is a constant $C_\alpha$ such that
$$
  \n{\Delta_\xi^\alpha b(\xi)}_{op} \leq
  C_\alpha \jp{\xi}^{\mu-\rho|\alpha|}
$$
holds for all $[\xi]\in\Gh$.
We record the following straightforward 
lemma that follows from the
smoothness of symbols in $x$ and the compactness of $G$:

\begin{lem}\label{LEM:frozensymbols}
We have $a\in \Ccl^m_{\rho,\delta}(G)$ if and only if
$a\in \Ccl^m_{\rho,\delta \#}(G)$.
\end{lem}

Another tool which will be required for the proof is the
finite version of the Leibniz formula which appeared in
\cite{RTW}.
Given a continuous unitary matrix representation
$\xi^0=\begin{bmatrix} \xi^0_{ij} \end{bmatrix}_{1\leq i,j\leq \ell}:
G\to\Bbb C^{\ell\times \ell}$,  
$\ell=\dim(\xi^0)$,
let $q(x)=\xi^0(x)-I$ (i.e. $q_{ij}=\xi^0_{ij}-\delta_{ij}$
with Kronecker's deltas $\delta_{ij}$), and define
$$
  \D_{ij}\widehat{f}(\xi) := \widehat{q_{ij} f}(\xi).
$$
In the previous notation, we could also 
write $\D_{ij}=\Delta_{q_{ij}}$.
For a multi-index $\gamma\in \N_0^{\ell^2}$, we write
$|\gamma|=\sum_{i,j=1}^\ell|\gamma_{ij}|$, and for higher order
difference operators we write
$\D^\gamma=\D_{11}^{\gamma_{11}}\D_{12}^{\gamma_{12}}\cdots
\D_{\ell,\ell-1}^{\gamma_{\ell,\ell-1}}
\D_{\ell\ell}^{\gamma_{\ell\ell}}$.
In contrast to the asymptotic 
Leibniz rule \cite[Thm.~10.7.12]{RT-book} for arbitrary
difference operators,
operators $\D$ satisfy the finite Leibniz formula:
\begin{prop}\label{prop:Ga-Leibniz}
For all $\gamma\in\N_0^{\ell^2}$ we have
$$
\D^\gamma(ab)=
\sum_{|\varepsilon|,|\delta|\leq |\gamma|\leq |\varepsilon|+|\delta|}
C_{\gamma\varepsilon\delta}\ (\D^\varepsilon a)\ 
(\D^\delta b),
$$
with the summation taken over all 
$\varepsilon,\delta\in\N_0^{\ell^2}$ satisfying
$|\varepsilon|,|\delta|\leq 
|\gamma|\leq |\varepsilon|+|\delta|$. 
In particular, for
$|\gamma|=1$, we have
\begin{eqnarray}\label{EQ:G-leinbiz-1}
\D_{ij}(ab) =
  \left(\D_{ij} a\right) b
  + a \left(\D_{ij} b\right)
  + \sum_{k=1}^\ell
  \left( \D_{ik} a \right)
  \left( \D_{kj} b \right).
\end{eqnarray} 
\end{prop}
Difference operators $\D$ lead to strongly admissible 
collections (see \cite{RTW}):
\begin{lem}\label{LEM:RTW-strong-adms}
The family of difference operators associated to the family of
functions
$\{q_{ij}=\xi_{ij}-\delta_{ij}\}_{[\xi]\in\Gh,\ 1\leq i,j\leq 
\dim(\xi)}$
is strongly admissible. Moreover, this family has a finite
subfamily 
associated to finitely many representations which is still 
strongly admissible.
\end{lem} 

We now give a simple proof of the equalities \eqref{EQ:Ga-multip-pos1}
and \eqref{EQ:Ga-multip-pos2}.

\begin{prop}\label{PROP:Ga-multip-pos}
We have 
$$
  \langle a\ast f,f\rangle_{L^2(G)}
  = \sum_{[\xi]\in\widehat{G}} {\rm dim}(\xi)
  \ {\rm Tr}\left(\widehat{f}(\xi) \
  \widehat{a}(\xi)\ \widehat{f}(\xi)^\ast
   \right),
$$
$$
  \langle f\ast a,f\rangle_{L^2(G)}
  = \sum_{[\xi]\in\widehat{G}} {\rm dim}(\xi)
  \ {\rm Tr}\left(\widehat{f}(\xi)^\ast\ \widehat{a}(\xi)
  \ \widehat{f}(\xi)  \right).
$$
\end{prop} 

\begin{proof}
The second claimed equality follows from the following calculation:
\begin{eqnarray*}
  && \langle f\ast a,f\rangle_{L^2(G)} 
   =  \int_G (f\ast a)(x)\ \overline{f(x)}\ {\rm d}x \\
  & = & \int_G \sum_{[\xi]\in\widehat{G}} {\rm dim}(\xi)
  \ {\rm Tr}\left( \xi(x)\ \widehat{a}(\xi)\ \widehat{f}(\xi) \right)
  \overline{\sum_{[\eta]\in\widehat{G}} {\rm dim}(\eta)
    \ {\rm Tr}\left( \eta(x)\ \widehat{f}(\eta) \right)}
  \ {\rm d}x \\
  & = & \int_G \sum_{[\xi]\in\widehat{G}} {\rm dim}(\xi)
  \sum_{k,l,m=1}^{{\rm dim}(\xi)} \xi(x)_{kl} \widehat{a}(\xi)_{lm}
  \widehat{f}(\xi)_{mk}
  \overline{\sum_{[\eta]\in\widehat{G}} {\rm dim}(\eta)
    \sum_{p,q=1}^{{\rm dim}(\eta)} \eta(x)_{pq} \widehat{f}(\eta)_{qp}}
  \ {\rm d}x \\
  & = & \sum_{[\xi]\in\widehat{G}} {\rm dim}(\xi)
  \sum_{k,l,m=1}^{{\rm dim}(\xi)} \widehat{a}(\xi)_{lm}
  \widehat{f}(\xi)_{mk} \overline{\widehat{f}(\xi)_{lk}} \\
  & = & \sum_{[\xi]\in\widehat{G}} {\rm dim}(\xi)
  \ {\rm Tr}\left( \widehat{a}(\xi)
  \ \widehat{f}(\xi)\ \widehat{f}(\xi)^\ast \right),
\end{eqnarray*}
where we used the orthogonality of the matrix elements of
the representations.
The first claimed equality can be proven in an analogous way.
\end{proof} 

We also record the Sobolev boundedness 
result that was Theorem 10.8.1 in \cite{RT-book}:

\begin{thm}
\label{THM:su2-Sobolev}
Let $G$ be a compact Lie group.
Let $A$ be a continuous linear operator 
from $C^\infty(G)$ to $C^\infty(G)$ and
let $\sigma_A$ be its symbol. 
Assume that
there exist constants $m,C_\alpha\in\R$ such that
$$
  \|\partial_x^\alpha\sigma_A(x,\xi)\|_{op}\leq C_\alpha\ 
  \jp{\xi}^m
$$
holds for all $x\in G$, $\xi\in{\rm Rep}(G)$, and 
all multi-indices $\alpha$.
Then
$A$ extends to a bounded operator from $H^s(G)$ to $H^{s-m}(G)$
for all $s\in\R$.
\end{thm}

\subsection{Amplitudes on $G$}

Let $0\leq\delta,\rho\leq 1$.
An amplitude $a\in{\cal A}_{\rho,\delta}^m(G)$ is a mapping
defined on $G\times G\times\Rep(G)$, smooth in
$x$ and $y$, such that
for an irreducible unitary representation 
$\xi:G\to{\cal U}({\cal H}_\xi)$ 
we have\footnote{Especially, if $\xi$ is a unitary matrix 
representation of dimension $d$,
then $a(x,y,\xi)\in\Bbb C^{d\times d}$}
linear operators
$$
  a(x,y,\xi):{\cal H}_\xi\to{\cal H}_\xi,
$$
and for a strongly admissible collection of difference
operators $\triangle_\xi^\alpha$
the amplitude satisfies the {\it amplitude inequalities}
$$
  \left\| \triangle_\xi^\alpha \partial_x^\beta \partial_y^\gamma
    a(x,y,\xi) \right\|_{op} \leq C_{\alpha\beta\gamma}
  \ \langle\xi\rangle^{m-\rho|\alpha|+\delta|\beta+\gamma|},
$$
for all multi-indices $\alpha,\beta,\gamma$
and for all $(x,y,[\xi])\in G\times G\times\Gh$.
For an amplitude $a$,
the {\it amplitude operator}
${\rm Op}(a):C^\infty(G)\to{\cal D}'(G)$ is defined by
\begin{equation}\label{EQ:Ga-def-amp}
 {\rm Op}(a)u(x) := \sum_{[\eta]\in\widehat{G}} {\rm dim}(\eta)
    \ {\rm Tr}\left( \eta(x) \int_G 
    a(x,y,\eta)\ u(y)\  \eta(y)^\ast\ {\rm d}y
      \right).
\end{equation} 
Notice that if here $a(x,y,\eta)=\sigma_A(x,\eta)$ 
then ${\rm Op}(a)=A$ as in \eqref{EQ:Ga-A-op-def}. This
definition can be justified as follows:

\begin{prop}
Let $0\leq\delta<1$ and $0\leq \rho\leq 1$, and let
$a\in{\cal A}_{\rho,\delta}^m(G)$.
Then ${\rm Op}(a)$ is a continuous linear operator
from $C^\infty(G)$ to $C^\infty(G)$.
\end{prop} 

\begin{proof}
By the definition of $\jp{\eta}$ we have
$(1-\Lap_G)\eta(y)=\jp{\eta}^2\eta(y)$. On the other hand, 
the Weyl spectral asymptotics formula for the Laplace
operator $\Lap_G$ implies that 
$\jp{\eta}^{-1}\leq C\dim(\eta)^{-2/\dim(G)}$ (see Proposition
10.3.19 in \cite{RT-book}). Consequently, integrating
by parts in the ${\rm d}y$-integral in \eqref{EQ:Ga-def-amp}
with operator $\jp{\eta}^{-2}(I-\Lap_G)$ arbitrarily many
times, we see that the $\eta$-series in
\eqref{EQ:Ga-def-amp} converges, so that $\Op(a)u\in C^\infty(G)$
provided that $u\in C^\infty(G)$. The continuity of
$\Op(a)$ on $C^\infty(G)$ follows by a similar argument.
\end{proof} 

\begin{rem}
In the proof we used the inequality 
$\dim(\eta)\leq C\jp{\eta}^{n/2}$, $n=\dim G$, which easily follows
from the Weyl spectral asymptotic formula (see Proposition
10.3.19 in \cite{RT-book}), and which is enough for the
purposes of the proof. However, a stronger inequality
$\dim(\eta)\leq C\jp{\eta}^{(n-l)/2}$ can be obtained
from the Weyl character formula, with $l=\rank G$.
For the details of this, see e.g. \cite[(11), (12)]{We}.
\end{rem} 

\begin{prop}\label{PROP:Ga-amps}
Let $0\leq\delta<\rho\leq 1$ and let
$a\in{\cal A}_{\rho,\delta}^m(G)$. Then $A=\Op(a)$ is
a pseudo-differential operator on $G$ with a matrix symbol
$\sigma_A\in\Ccl^m_{\rho,\delta}(G)$. Moreover,
$\sigma_A$ has the asymptotic expansion
$$
\sigma_A(x,\xi)
   \sim 
  \sum_{\alpha\geq 0} \frac{1}{\alpha!} \left.\partial_y^{(\alpha)}
  \triangle_\xi^\alpha a(x,y,\xi)\right|_{y=x}.
$$
\end{prop} 

\begin{proof}
If $\sigma_A$ is the matrix symbol of the continuous linear operator
$A=\Op(a):C^\infty(G)\to C^\infty(G)$, we can find it from
the formula $\sigma_A(x,\xi)=\xi(x)^* (A\xi)(x)$. By fixing
some basis in the representation spaces, we have
\begin{eqnarray*}
  && \sigma_A(x,\xi)_{mn} =  \sum_{l=1}^{{\rm dim}(\xi)}
  \xi(x^{-1})_{ml}\ \left( A\xi_{ln} \right)(x) \\
  & = & \sum_{l=1}^{{\rm dim}(\xi)}
  \xi(x^{-1})_{ml} \int_G
    \sum_{[\eta]\in\widehat{G}} {\rm dim}(\eta)\ {\rm Tr}\left(
      \eta(x)\ a(x,y,\eta)\ \xi(y)_{ln}\ 
      \eta(y)^\ast\right) {\rm d}y \\
  & = & \int_G \xi(x^{-1}y)_{mn} \sum_{[\eta]\in\widehat{G}}
  {\rm dim}(\eta)\ {\rm Tr}\left( \eta(y^{-1}x)\ a(x,y,\eta)\right) {\rm d}y \\
  & = & \int_G \xi(x^{-1}y)_{mn} \sum_{[\eta]\in\widehat{G}}
  {\rm dim}(\eta) \sum_{j,k=1}^{{\rm dim}(\eta)} \eta(y^{-1}x)_{jk}
  \ a(x,y,\eta)_{kj}\ {\rm d}y \\
  & = & \int_G \xi(z^{-1})_{mn} \sum_{[\eta]\in\widehat{G}}
  {\rm dim}(\eta) \sum_{j,k=1}^{{\rm dim}(\eta)} \eta(z)_{jk}
  \ a(x,xz^{-1},\eta)_{kj}\ {\rm d}z \\
  & \sim &
  \sum_{\alpha\geq 0} \frac{1}{\alpha!}\ \partial_u^{(\alpha)}
  \sum_{[\eta]\in\widehat{G}} {\rm dim}(\eta) 
  \sum_{j,k=1}^{{\rm dim}(\eta)}
  \left. a(x,u,\eta)_{kj}\right|_{u=x}
  \int_G \xi(z^{-1})_{mn}\ \eta(z)_{jk}\ q_\alpha(z)\ {\rm d}z,
\end{eqnarray*}
by the Taylor expansion \eqref{EQ:RTW-Taylor-exp}.
Using difference operators
$\triangle_\xi^\alpha\widehat{s}(\xi):=\widehat{q_\alpha s}(\xi)$,
we find
\begin{eqnarray*}
  &&
  \sum_{[\eta]\in\widehat{G}} {\rm dim}(\eta) 
  \sum_{j,k=1}^{{\rm dim}(\eta)}
  a(x,u,\eta)_{kj}
  \int_G \xi(z^{-1})\ \eta(z)_{jk}\ q_\alpha(z)\ {\rm d}z \\
  & = & 
  \int_G \xi(z)^\ast\ q_\alpha(z)
  \sum_{[\eta]\in\widehat{G}} {\rm dim}(\eta)\ {\rm Tr}\left(
     \eta(z)\ a(x,u,\eta) \right) {\rm d}z 
    =  \triangle_\xi^\alpha a(x,u,\xi).
\end{eqnarray*}
Thus
\begin{eqnarray*}
  \sigma_A(x,\xi)
  & \sim &
  \sum_{\alpha\geq 0} \frac{1}{\alpha!}\ \partial_u^{(\alpha)}
  \int_G \xi(z)^\ast\ q_\alpha(z)
  \left. \sum_{[\eta]\in\widehat{G}} {\rm dim}(\eta)
  \ {\rm Tr}\left(\eta(z)\ a(x,u,\eta) \right)\ 
  {\rm d}z \right|_{u=x} \\
  & = & \sum_{\alpha\geq 0} \frac{1}{\alpha!} \left.\partial_u^{(\alpha)}
  \triangle_\xi^\alpha a(x,u,\xi)\right|_{u=x}.
\end{eqnarray*}
The remainder in this asymptotic expansion
can be dealt with in a way similar to the argument for
the composition formulae, so we omit the proof.
\end{proof}

\subsection{Properties of even and odd functions}

On a group $G$,
function $f:G\to\mathbb C$ is called {\em even}
if it is inversion-invariant, i.e. if $f(x^{-1})=f(x)$ 
for every $x\in G$.
Function $f:G\to\mathbb C$ is called {\em odd}
if $f(x^{-1})=-f(x)$ for every $x\in G$.
Recall that $f:G\to\mathbb C$ is {\em central} if
$
  f(xy)=f(yx)
$
for all $x,y\in G$.
Linear combinations of characters
$\chi_\xi=(x\mapsto{\rm Tr}(\xi(x)))$
of irreducible unitary representations $\xi$ of a compact group $G$
are central, and such linear combinations are dense
among the central functions of $C(G)$.
When $G$ is a compact Lie group,
for $Y\in\mathfrak g$ and $f\in C^\infty(G)$ we define
\begin{eqnarray*}
  L_Y f(x)  :=  \frac{{\rm d}}{{\rm d}t} f(x\exp(tY))|_{t=0}, \quad
  R_Y f(x)  :=  \frac{{\rm d}}{{\rm d}t} f(\exp(tY)x)|_{t=0},
\end{eqnarray*}
so that $L_Y,R_Y$ are the first order differential operators,
$L_Y$ being left-invariant and $R_Y$ right-invariant.
For a central function $f$ we have $L_Y f=R_Y f$,
which would not be true for an arbitrary smooth function $f$.
Moreover, if $f$ is even and central
then
\begin{eqnarray*}
  L_Y f(x^{-1})
  & = & - L_Y f(x),
\end{eqnarray*}
i.e. $L_Y f$ is odd in this case.
Similarly $L_Y f$ is even for odd central functions $f$,
but $L_Y f$ does not have to be central.
More precisely, for central $f\in C^\infty(G)$ we obtain
\begin{eqnarray*}
  L_Y f(u^{-1}xu)
  & = & L_{uYu^{-1}}f(x),
\end{eqnarray*}
where $u\in G$.
For higher order derivatives of even and odd functions, 
taking the differential of
$$
f(x \exp(t_1 X_1)\ldots \exp(t_k X_k))=
\pm f(x^{-1} \exp((-t_k) X_k)\ldots \exp((-t_1) X_1))
$$
at $t_1=\cdots=t_k=0$,
we obtain

\begin{prop}\label{PROP:Ga-even-odd}
Let $f\in C^\infty(G)$ be even and central,
and $X_1,\cdots,X_k\in\mathfrak{g}$. 
Then
\begin{eqnarray*}
  L_{X_1} L_{X_2}\cdots L_{X_{k-1}} L_{X_k} f(x^{-1}) 
  =  (-1)^k 
  L_{X_k} L_{X_{k-1}}\cdots L_{X_2} L_{X_1} f(x).
\end{eqnarray*}
Similarly, if $f\in C^\infty(G)$ is an odd central function,
then we have the equality
\begin{eqnarray*}
  L_{X_1} L_{X_2}\cdots L_{X_{k-1}} L_{X_k} f(x^{-1}) 
  =  (-1)^{k+1} 
  L_{X_k} L_{X_{k-1}}\cdots L_{X_2} L_{X_1} f(x).
\end{eqnarray*}
\end{prop}


\section{Proof of the sharp G{\aa}rding inequality}
\label{SEC:Ga-proofs}

We notice that
if a linear operator $Q:H^{(m-1)/2}(G)\to H^{-(m-1)/2}(G)$ is bounded
then
\begin{eqnarray*}
  \Repa \p{Qu,u}_{L^2}
  & \geq & -\left|\p{Qu,u}_{L^2}\right| \\
  & \geq & -\|Qu\|_{H^{-(m-1)/2}}\ \|u\|_{H^{(m-1)/2}} \\
  & \geq & -\|Q\|_{{\cal L}(H^{(m-1)/2},H^{-(m-1)/2})}\ \|u\|_{H^{(m-1)/2}}^2.
\end{eqnarray*}
Hence Theorem~\ref{THM:Garding} would follow if we could show that
$A=P+Q$,
where $P$ is positive (on $C^\infty(G)\subset L^2(G)$)
and $Q:H^{(m-1)/2}(G)\to H^{-(m-1)/2}(G)$ is bounded.
The proof of this decomposition will be done in several steps.

\subsection{Construction of $w_\xi$}
\label{SEC:wxi}

First, we construct an auxiliary function $w_\xi$ which will
play a crucial role for our proof. 

We can treat $G$ as a closed subgroup of
${\rm GL}(N,\Bbb R)\subset\Bbb R^{N\times N}$ 
for some $N\in\Bbb N$.
Then its Lie algebra $\mathfrak{g}\subset\Bbb R^{N\times N}$
is an $n$-dimensional vector subspace
(hence identifiable with $\Bbb R^n$)
such that $[A,B]:=AB-BA\in\mathfrak{g}$ for every $A,B\in\mathfrak{g}$.
Let $U\subset G$ be a neighbourhood of the 
neutral element $e\in G$,
and let $V\subset {\mathfrak g}$ be a neighbourhood of
$0\in\mathfrak{g}\cong\Bbb R^n$,
such that the matrix exponential mapping is a 
diffeomorphism $\exp:V\to U$.

For the construction and for the 
notation only in Section \ref{SEC:wxi},
we define the central norm $|\cdot|$ on ${\mathfrak g}$ as 
follows\footnote{In fact, any central norm 
$|\cdot|$ on ${\mathfrak g}$ will work.}.
Take the Euclidean norm $|\cdot|_0$ on ${\mathfrak g}$ and define
\begin{equation}\label{EQ:inorm}
|X|=\int_G |u X u^{-1}|_0 \drm u,
\end{equation} 
where we may view the product under the integral as the product of
matrices in $\R^{N\times N}$. Then by definition the norm
\eqref{EQ:inorm} is invariant by the adjoint representation, and
we have, in particular $|\exp^{-1}(xy)|=|\exp^{-1}(yx)|$, etc.

We may assume that $V$ is the open ball
$V=\Bbb B(0,r)=\{Z\in\Bbb R^n:\ |Z|<r\}$ of radius $r>0$.
Let $\phi:[0,r)\to [0,\infty)$ be a smooth function such that
$(Z\mapsto\phi(|Z|)):\mathfrak{g}\to\Bbb R$ is supported in $V$
and $\phi(s)=1$ for small $s>0$.
For every $\xi\in \Rep(G)$ we define
\begin{equation}\label{EQ:Ga-def-wxi}
  w_\xi(x) :=
  \phi(|\exp^{-1}(x)|\jp{\xi}^{1/2}) 
  \ \psi(\exp^{-1}(x))\ \jp{\xi}^{n/4},
\end{equation} 
where 
$$\psi(Y)=C_0 
\left|\det D\exp (Y)\right|^{-1/2} f(Y)^{-1/2},$$ 
$D\exp$ is 
the Jacobi matrix of $\exp$, $f(Y)$ 
is the density with respect to the Lebesgue
measure of the Haar measure on $G$ pulled back to
$\mathfrak g\cong\Rn$ by the exponential mapping, and with 
constant $C_0=\p{\int_\Rn \phi(|Z|)^2 \drm Z}^{-1/2}$. 
By $I_{\dim(\xi)}$ we
denote the identity mapping on $\C^{\dim(\xi)}$.
For $x,y\in G$ close to each other, 
${\rm dist}(x,y)$ is the geodesic distance between
$x$ and $y$.

\begin{lem}\label{LEM:Ga-wxi}
We have $w_\xi \in C^\infty(G)$, $w_\xi(e)=C_0\jp{\xi}^{n/4}$, 
$w_\xi$ is central and inversion-invariant,
i.e. $w_\xi(xy)=w_\xi(yx)$ and $w_\xi(x^{-1})=w_\xi(x)$ 
for every $x,y\in G$. Also,
${\rm dist}(x,e)\approx |\exp^{-1}(x)|\leq r\jp{\xi}^{-1/2}$ 
on the support of $w_\xi$. Moreover, 
$\left\|w_\xi\right\|_{L^2(G)}=1$ for all 
$\xi\in\Rep(G)$.
Finally, we have 
$\left((x,\xi)\mapsto w_\xi(x) I_{{\rm dim}(\xi)}\right)
\in\Ccl^{n/4}_{1,1/2}(G)$.
\end{lem} 

\begin{proof}
It is easy to see that 
$w_\xi \in C^\infty(G)$, $w_\xi(e)=C_0\jp{\xi}^{n/4}$, 
and that $w_\xi$ is inversion-invariant. 
Clearly
${\rm dist}(x,e)\approx |\exp^{-1}(x)|\leq r\jp{\xi}^{-1/2}$ 
on the support of $w_\xi$
in view of properties of the function $\phi$.
In particular, \eqref{EQ:Ga-def-wxi} is
well-defined and ${\rm supp }\ w_\xi\subset U$.
From \eqref{EQ:inorm} it also follows that $w_\xi$ is central
since $f$ is invariant under adjoint representation as a density
of two bi-invariant measures.

Let us now show that 
$\left\|w_\xi\right\|_{L^2(G)}=1$ for all $\xi\in\Rep (G)$.
Indeed,
\begin{eqnarray*}
  \int_G |w_\xi(x)|^2 \drm x
 & = & \jp{\xi}^{n/2} \int_\Rn \phi(|Y|\jp{\xi}^{1/2})^2 |\psi(Y)|^2
  \abs{\det D\exp(Y)}\ f(Y) \drm Y \\
 & = &C_0^2 \int_\Rn \phi(|Z|)^2\ {\rm d}Z,
\end{eqnarray*}
so that $\|w_\xi\|_{L^2(G)}=1$ in view of the
choice of the constant $C_0$.
Thus, the main thing is to check that  
$w_\xi I_{{\rm dim}(\xi)}\in\Ccl^{n/4}_{1,1/2}(G)$.
By Lemma~\ref{LEM:frozensymbols},
we need to check that for every multi-index
$\beta$ and every $x_0\in G$ we have
$\partial_x^\beta w_\xi(x_0)\in \Ccl^{n/4+|\beta|/2}_{1 \#}(G)$.
We observe that the $x$-derivatives of $w_\xi$ are sums
of terms of the form
\begin{equation}\label{EQ:G:terms}
  \chi(\exp^{-1}(x))\ \widetilde{\phi}(|\exp^{-1}(x)|\jp{\xi}^{1/2})
  \ \jp{\xi}^{n/4+l/2}\ I_{\dim(\xi)},
\end{equation} 
where
$\chi\in C_0^\infty(V)$,
$\widetilde{\phi}\in C_0^\infty(\R)$, 
$\widetilde{\phi}$ is constant near the origin, and $l$ is
an integer such that $0\leq l\leq |\beta|$.
We note that $\jp{\xi}^{n/4+l/2}I_{\dim(\xi)}$ 
is the symbol of the
pseudo-differential operator $(1-\Lap_G)^{n/8+l/4}$, and hence
$\jp{\xi}^{n/4+l/2}I_{\dim(\xi)}\in \Ccl^{n/4+l/2}_{1\#}
\subset\Ccl^{n/4+|\beta|/2}_{1\#}$.
Moreover, we can eliminate
it from the formulae by the composition formulae for
the matrix-valued symbols (see \cite[Thm.~10.7.9]{RT-book}).
Thus we have to check that for every $x_0\in G$, the other terms
in \eqref{EQ:G:terms} fixed at $x=x_0$
are in 
$\Ccl^{0}_{1 \#}(G)$, i.e. that
\begin{equation}\label{EQ:G:terms2}
  \widetilde{\phi}(|\exp^{-1}(x_0)|\jp{\xi}^{1/2})\ I_{\dim(\xi)}
  \in\Ccl^{0}_{1 \#}(G).
\end{equation}
If $\exp^{-1}(x_0)=0$, then this
symbol is a constant times the identity $I_{\dim(\xi)}$
and hence it is in $\Ccl^{0}_{1 \#}(G)$.
On the other hand, if $\exp^{-1}(x_0)\not=0$, then
the symbol \eqref{EQ:G:terms2} is compactly supported
in $\xi$, and hence defines a smoothing operator.
Indeed, in this case it has decay of any order in
$\jp{\xi}$, together with all difference operators applied
to it, with constants depending on $x_0$, so it is smoothing
by Theorem \ref{thm:main}.

Let us also give an alternative argument relating this operator
to a corresponding operators on $\mathfrak g$.
Writing
$\varphi_v(t):=\widetilde{\phi}(|v|t)$ and using
the characterisation of pseudo-differential operators 
in Theorem \ref{thm:main}, we notice that \eqref{EQ:G:terms2}
holds if for all $x_0\in G$, the operators
$\varphi_{\exp^{-1}(x_0)}((I-\Lap_G)^{1/4})$
belong to $\Psi^0(G)$. Looking at these operators
locally near every point $x\in G$ and introducing
$\theta\in C_0^\infty(\Bbb R^n)$ such that 
$\theta\circ\exp_x^{-1}$
is supported in a small neighbourhood
near $x$, with $\exp_x:=(Z\mapsto x\exp(Z)):\mathfrak{g}\to G$
the exponential mapping centred at $x$, 
we have to show that
\begin{equation}\label{EQ:G-ps1}
  \theta(y) \varphi_{v}(B)\in \Psi^0(\Rn)
\end{equation} 
holds locally on the support of $\theta$,
for all $v=\exp^{-1}(x_0)$, where operator $B$
is the pullback by $\exp_x$ of the operator
$(I-\Lap_G)^{1/4}$ near $x$. In particular, we have
$B\in \Psi^{1/2}_{1,0}(\Rn)$, $B$ is elliptic on
the support of $\theta$, and its symbol is real-valued.

We now observe that if $v=0$, then the operator in 
\eqref{EQ:G-ps1} is the multiplication operator by a
smooth function, so that \eqref{EQ:G-ps1} is true in
this case. If $v\not=0$, we can show that the
operator in \eqref{EQ:G-ps1} is actually a smoothing
operator, so that \eqref{EQ:G-ps1} is also true.
Here
$\va_v\in C_0^\infty(\R)$ since $v\not=0$.
We denote $D_t=\frac{1}{{\rm i}2\pi}\partial_t$.
Let $f\in L^2(\Rn)$ be compactly supported, and 
let $u=u(t,x)$ be the solution to the Cauchy problem
$$D_t u=Bu, \ u(0,\cdot)=f.$$
We can write $u(t,\cdot)={\rm e}^{{\rm i}2\pi tB}f$ and we have
$u(t,\cdot)\in L^2(\Bbb R^n)$. Consequently,
$$
  \varphi_v(B)f=\int_\R \p{{\rm e}^{{\rm i}2\pi tB} f}\ 
  \widehat{\varphi_v}(t)\ {\rm d}t
  = \int_\R B^{-k} u(t,\cdot)\ D_t^k 
  \widehat{\varphi_v}(t)\ {\rm d}t,
$$
where we integrated by parts $k$ times using the relation
$u=B^{-1} D_t u$, and where we can localise to a neighbourhood
of a point $x$ at each step. Consequently, we obtain that
$\varphi_v(B) f\in H^{k/2}_{loc}(\Bbb R^n)$ for all 
$k\in\Bbb Z^+$, so that actually
$\varphi_v(B)f\in C^\infty(\Bbb R^n)$. Thus, the operator
$\varphi_v(B)$ is smoothing and \eqref{EQ:G-ps1} holds also for
$v\not=0$.
\end{proof}

\subsection{Auxiliary positive operator $P$}
We now introduce a positive operator $P$ which will 
be important for the proof of the sharp
G{\aa}rding inequality. This operator $P$ will give
a positive approximation to our operator $A$.

\begin{prop}\label{PROP:Ga-op-P}
Let $\sigma_A\in\Ccl^m_{1,0}(G)$.
Let us define an amplitude $p$ by
\begin{equation}\label{EQ:a-amplitude}
  p(x,y,\xi) := \int_G w_\xi(xz^{-1})\ w_\xi(yz^{-1})\ 
  \sigma_A(z,\xi) 
  \ {\rm d}z,
\end{equation}
where $w_\xi\in C^\infty(G)$ is as in \eqref{EQ:Ga-def-wxi}.
Let the amplitude operator $P={\rm Op}(p)$ be given by
\begin{eqnarray*}
  Pu(x) =  \int_G \sum_{[\xi]\in\widehat{G}} {\rm dim}(\xi)
  \ {\rm Tr}\left( \xi(y^{-1}x)\ p(x,y,\xi) \right)\ u(y)\ 
  {\rm d}y.
\end{eqnarray*}
Then $p\in {\cal A}^m_{1,1/2}(G)$ and the 
operator $P$ is positive. 
\end{prop}

\begin{proof}
We observe that
\begin{multline*}
\|p(x,y,\xi)\|_{op}\leq 
\int_G |w_\xi(xz^{-1})\ w_\xi(yz^{-1})| \drm z\
\p{\sup_{z\in G}\|\sigma_A(z,\xi)\|_{op}} 
\leq
C\jp{\xi}^m
\end{multline*} 
because $\|w_\xi\|^2_{L^2(G)}=1$ by Lemma \ref{LEM:Ga-wxi}.
Then $p\in {\cal A}^m_{1,1/2}(G)$ follows from
Lemma~\ref{LEM:Ga-wxi} and the Leibniz formula in
Proposition~\ref{prop:Ga-Leibniz}
by an argument similar to the one which will be given  
in the proof of Lemma~\ref{LEM:Ga-difaux-2}, so we omit it.
Let $(e_k)_{k=1}^\ell$ be an orthonormal basis for $\Bbb C^\ell$.
For matrices 
$M,Q\in\Bbb C^{\ell\times \ell}$, where $Q$ is positive,
we have
\begin{equation}\label{EQ:Ga-mat-pos}
  {\rm Tr}(M^\ast Q M)
  = \sum_{k=1}^\ell \langle M^\ast Q M e_k, e_k 
  \rangle_{\Bbb C^\ell}
  = \sum_{k=1}^\ell \langle Q M e_k, M e_k \rangle_{\Bbb C^\ell} 
  \geq 0.
\end{equation} 
Let us denote 
$$
M(z,\xi):= \int_G w_\xi(yz^{-1})\ \xi(yz^{-1})^\ast\ u(y)\ {\rm d}y.
$$
We can now show that the operator $P$ is positive:
\begin{eqnarray*}
  && \langle Pu,u\rangle_{L^2(G)} 
   =  \int_G Pu(x)\ \overline{u(x)}\ {\rm d}x \\
  & = & \int_G \int_G \sum_{[\xi]\in\widehat{G}} {\rm dim}(\xi)
  \ {\rm Tr}\left( \xi(x)\ p(x,y,\xi)\ u(y)\ \xi(y)^\ast \right)
  \ {\rm d}y\ \overline{u(x)}\ {\rm d}x \\
  & = & \int_G \sum_{[\xi]\in\widehat{G}} {\rm dim}(\xi)
  \int_G
  {\rm Tr}\left( \xi(x) \int_G w_\xi(xz^{-1}) w_\xi(yz^{-1})
    \sigma_A(z,\xi) {\rm d}z\, u(y) \xi(y)^\ast {\rm d}y \right)
  \overline{u(x)} {\rm d}x \\
  & = & \int_G \sum_{[\xi]\in\widehat{G}} {\rm dim}(\xi)
  \ {\rm Tr}\left( M(z,\xi)^\ast\ \sigma_A(z,\xi)\ 
  M(z,\xi) \right)\ {\rm d}z,
\end{eqnarray*}
which is non-negative because of \eqref{EQ:Ga-mat-pos}.
\end{proof} 

\subsection{The difference $p(x,x,\xi)-\sigma_A(x,\xi)$}

In the earlier notation,
we show here that $p(x,x,\xi)-\sigma_A(x,\xi)$
is a symbol of a bounded operator from 
$H^s(G)$ to $H^{s-(m-1)}(G)$.

\begin{lem} \label{LEM:Ga-difaux-1}
Let $s\in\R$. Then
the pseudo-differential operator with the symbol
$p(x,x,\xi)-\sigma_A(x,\xi)$
is bounded from 
$H^s(G)$ to $H^{s-(m-1)}(G)$.
\end{lem} 

\begin{proof}
By Theorem~\ref{THM:su2-Sobolev} it is enough to show that
$$
  \left\| \partial_x^\beta 
  \left( p(x,x,\xi)-\sigma_A(x,\xi) \right) \right\|_{op}
  \leq C_\beta \langle\xi\rangle^{m-1}
$$
holds for every multi-index $\beta$.
By Lemma~\ref{LEM:Ga-wxi} we have
$$
  \partial_x^\beta \left( p(x,x,\xi)-\sigma_A(x,\xi) \right)
  = \int_G w_\xi(z)^2 \left(
  \partial_x^\beta \sigma_A(xz^{-1},\xi)
    - \partial_x^\beta\sigma_A(x,\xi) \right)\ {\rm d}z. 
$$
We notice that ${\rm dist}(z,e)\leq C\jp{\xi}^{-1/2}$ on the
support of $w_\xi$, and we can use the Taylor expansion of
$\partial_x^\beta\sigma_A(xz^{-1},\xi)$ at $x$ to get
\begin{equation}\label{EQ:Ga-symb-T-exp}
\partial_x^\beta\sigma_A(xz^{-1},\xi)=
\partial_x^\beta\sigma_A(x,\xi)+
\sum_{|\gamma|=1}\partial_x^{(\gamma)}
\partial_x^\beta\sigma_A(x,\xi)
q_\gamma(z)+O({\rm dist}(z,e)^2).
\end{equation} 
Taking the Taylor polynomials $q_\gamma$ to be odd,
$q_\gamma(z)=-q_\gamma(z^{-1})$, and using the evenness of
$w_\xi$ from Lemma~\ref{LEM:Ga-wxi}, 
we can conclude that $\int_G w_\xi(z)^2 \ q_\gamma(z)\drm z=0$.
Since for all $\beta$ and $\gamma$ we have 
$\left\|\partial_x^{(\gamma)}\partial_x^\beta
\sigma_A(x,\xi)\right\|_{op}
\leq C\jp{\xi}^m$, we can estimate
$$
\n{\partial_x^\beta 
\left( p(x,x,\xi)-\sigma_A(x,\xi) \right)}_{op}\leq 
C\jp{\xi}^m\sum_{|\gamma|=2} 
\int_G  w_\xi(z)^2 |q_\gamma(z)|\drm z\leq C\jp{\xi}^{m-1}
$$
because $|q_\gamma(z)|\leq C\jp{\xi}^{-1}$ on the support
of $w_\xi$, for $|\gamma|=2$. 
\end{proof} 

\subsection{The difference $\sigma_P(x,\xi)-p(x,x,\xi)$}
Let $\sigma_P$ be the matrix symbol of the operator
$P$ from Proposition~\ref{PROP:Ga-op-P}.

\begin{lem} \label{LEM:Ga-difaux-2}
Let $s\in\R$. Then
the pseudo-differential operator with the symbol
$\sigma_P(x,\xi)-p(x,x,\xi)$
is bounded from 
$H^s(G)$ to $H^{s-(m-1)}(G)$.
\end{lem} 

\begin{proof}
Observe that for a fixed $s\in \R$,
it is enough to take sufficiently many
derivatives (and not infinitely many) for the
Sobolev boundedness in Theorem~\ref{THM:su2-Sobolev}.
Thus it is enough to prove that 
for sufficiently many $\beta\in\Bbb N_0^n$ it holds that
$$
  \left\| \partial_x^\beta(\sigma_P(x,\xi) - p(x,x,\xi)) \right\|_{op}
  \leq C_\beta \langle\xi\rangle^{m-1}.
$$
By an argument in the proof of 
Proposition~\ref{PROP:Ga-amps} we have the 
expansion
$$
  \sigma_P(x,\xi) \sim \sum_{\alpha\geq 0} \frac{1}{\alpha!}
  \ \left.\triangle_\xi^\alpha \partial_y^{(\alpha)} 
  p(x,y,\xi)\right|_{y=x},
$$
whose asymptotic properties we will discuss below.
Instead of studying the terms
$
  \partial_x^\beta 
  \p{\left.  \triangle_\xi^\alpha \partial_y^{(\alpha)}
    p(x,y,\xi)\right|_{y=x}},
$
we may study
$
  \partial_x^\beta \p{\left. 
  \triangle_\xi^\alpha \partial_y^\alpha
  p(x,y,\xi) \right|_{y=x}}
$
as well.
Moreover, abusing the notation slightly,
without loss of generality we can look only at
the right-invariant derivatives $\partial_y^\alpha$ and
left-invariant derivatives $\partial_x^\beta$.
Recalling that
\begin{eqnarray*}
  p(x,y,\xi)  = 
  \int_G w_\xi(xz^{-1})\ w_\xi(yz^{-1})\ 
  \sigma_A(z,\xi)\ {\rm d}z, 
\end{eqnarray*}
we notice that
\begin{equation}\label{EQ:Ga-mult-int0}
   \partial_x^\beta\p{\left. \triangle_\xi^\alpha 
  \partial_y^{\alpha} p(x,y,\xi)
    \right|_{y=x} }
   =  \triangle_\xi^\alpha
  \int_G w_\xi(z)\ (\partial_z^\alpha w_\xi)(z)
  \ \partial_x^\beta\sigma_A(z^{-1}x,\xi)\ {\rm d}x.
\end{equation}
We notice also that by Remark~\ref{REM:Ga-rhodel-classes}
we can replace differences 
$\Delta_\xi$ by $\D_\xi$ with a suitable 
correction for multi-indices.
The application of $\D^\alpha$ here
introduces (due to the Leibniz formula
in Proposition~\ref{prop:Ga-Leibniz}) a 
finite sum of terms of the type
\begin{equation}\label{EQ:Ga-mult-int}
  \int_G \left(\D_\xi^\kappa w_\xi(z)\right)
  \left(\D_\xi^\lambda \partial_z^\alpha w_\xi(z)\right)
  \ \left( \D_\xi^\mu \sigma_A(z^{-1}x,\xi) \right) {\rm d}z,
\end{equation} 
where $|\kappa+\lambda+\mu|\geq|\alpha|$.
Recalling that $w_\xi\in\Ccl^{n/4}_{1,1/2}$ by Lemma~\ref{LEM:Ga-wxi},
we get that
$$
\abs{\left(\D_\xi^\kappa w_\xi(z)\right)
  \left(\D_\xi^\lambda \partial_z^\alpha w_\xi(z)\right)
  \ \left( \D_\xi^\mu \sigma_A(z^{-1}x,\xi) \right)}\leq 
  C\jp{\xi}^{m+n/2-|\alpha|/2}.
$$
Taking into account that the support of $z\mapsto w_\xi(z)$ 
is contained in
the set of measure $C\jp{\xi}^{-n/2}$ by Lemma~\ref{LEM:Ga-wxi},
and that taking differences in $\xi$
does not increase the support in $z$,
we get that the integral in \eqref{EQ:Ga-mult-int} can be estimated
by $C\jp{\xi}^{m-|\alpha|/2}$. Thus, we get
\begin{equation}\label{EQ:Ga-mult-int2}
  \n{ \partial_x^\beta \p{\left. \triangle_\xi^\alpha 
  \partial_y^{\alpha} p(x,y,\xi)
    \right|_{y=x}} }_{op} \leq C\jp{\xi}^{m-|\alpha|/2}.
\end{equation}
For $|\alpha|\geq 2$ this implies the desired bound by
$C\jp{\xi}^{m-1}$ for the Sobolev boundedness of the 
corresponding operator. 
Now, assume that $|\alpha|=1$.
Taking the Taylor expansion of $\sigma_A(z^{-1}x,\xi)$ at $x$ similar
to the one in \eqref{EQ:Ga-symb-T-exp} we see that the first term
vanishes:
$$
\int_G w_\xi(z) \ (\partial_z^\alpha w_\xi)(z) \drm z =0
$$
for $|\alpha|=1$ because functions $w_\xi$ 
and $\partial_z^\alpha w_\xi$
are even and odd, respectively, see
Proposition~\ref{PROP:Ga-even-odd}. 
Consequently, for $|\gamma|\geq 1$, 
we can estimate
$$
\abs{w_\xi(z) \ (\partial_z^\alpha w_\xi)(z)\ q_\gamma(z)}
\leq C\jp{\xi}^{n/2+|\alpha|/2-|\gamma|/2},
$$
which together with \eqref{EQ:Ga-mult-int0} gives
$$\n{\partial_x^\beta \p{\left. \triangle_\xi^\alpha 
  \partial_y^{\alpha} p(x,y,\xi)
    \right|_{y=x}} }_{op} \leq C\jp{\xi}^{m-|\alpha|/2-|\gamma|/2}
    \leq C\jp{\xi}^{m-1}
$$
because $|\alpha|=1$ and $|\gamma|\geq 1$.
Finally, let us look at the remainder
$$
  \sigma_{R_N}(x,\xi)
  = \sigma_P(x,\xi) - \sum_{|\alpha|<N} \frac{1}{\alpha!}
  \ \triangle_\xi^\alpha \left.\partial_y^{(\alpha)} p(x,y,\xi)\right|_{y=x}.
$$
By the arguments similar to the above we can see that 
$$
  \left\| \partial_x^\beta \sigma_{R_N}(x,\xi) \right\|_{op}
  \leq C_\beta \langle\xi\rangle^{m+n/2+|\beta|/2-N/2},
$$
so that for every $s,t\in\Bbb R$
there exists a sufficiently large $N_{st}$ such that 
$R_N$ is bounded from $H^s(G)$ to $H^t(G)$
whenever $N\geq N_{st}$.
This concludes the proof.
\end{proof} 

\subsection{Proof of Theorem~\ref{THM:Garding}}
Let $Q=A-P$ with operator $P$ as in Proposition~\ref{PROP:Ga-op-P}.
Let $u\in C^\infty(G)$.
Then $A=P+Q$ and the positivity of $P$ implies
$$
\Repa (Au,u)_{L^2(G)}=\Repa (Pu,u)_{L^2(G)}+\Repa (Qu,u)_{L^2(G)}\geq
\Repa (Qu,u)_{L^2(G)}.
$$
Let now $P_0=\Op(p(x,x,\xi))$.
Writing $Q=(A-P_0)+(P_0-P)$, we have
$$\sigma_{A-P_0}(x,\xi)=\sigma_A(x,\xi)-p(x,x,\xi) 
\;\textrm{ and }\;
\sigma_{P_0-P}(x,\xi)=p(x,x,\xi)-\sigma_P(x,\xi).$$
Consequently, $A-P_0$ and $P_0-P$ are
bounded from $H^{(m-1)/2}(G)$ to $H^{-(m-1)/2}(G)$ by
Lemma~\ref{LEM:Ga-difaux-1} and Lemma~\ref{LEM:Ga-difaux-2},
respectively. Hence
$Q$ is bounded from $H^{(m-1)/2}(G)$ to $H^{-(m-1)/2}(G)$, so that
$$
  |\Repa(Qu,u)_{L^2(G)}|\leq \left\|Qu\right\|_{H^{-(m-1)/2}(G)}
  \left\|u\right\|_{H^{(m-1)/2}(G)}
  \leq C\left\|u\right\|^2_{H^{(m-1)/2}(G)},
$$
completing the proof of Theorem~\ref{THM:Garding}.

\subsection{Proof of Corollary~\ref{COR:Ga-Taylorthm}}
We note that the assumption
$\left\|\sigma_A(x,\xi)\right\|_{op}\leq C$ implies
that for any $\theta\in\R$ be have the inequality
$\Repa(C-\erm^{\irm\theta}\sigma_A(x,\xi))\geq 0$.
Consequently, the sharp G{\aa}rding inequality in
Theorem~\ref{THM:Garding} implies that we have
$$
\Repa((C-\erm^{\irm\theta}A)u,u)_{L^2(G)}\geq 
-C'\left\|u\right\|_{L^2(G)}^2
$$
for all $u\in L^2(G)$. From this it follows that
$\Repa(\erm^{\irm\theta}(Au,u)_{L^2(G)})\leq 
C''\left\|u\right\|_{L^2(G)}^2$,
so that $|(Au,u)_{L^2(G)}|\leq C''\left\|u\right\|_{L^2(G)}^2$, 
completing the proof of Corollary~\ref{COR:Ga-Taylorthm}.

\subsection{Proof of Corollary~\ref{COR:Ga-Kgothm}}
Let us define
$$B(x,\xi)=M^2\jp{\xi}^{2m+2s} I_{\dim\xi}-
\sigma_A(x,\xi)^*\sigma_A(x,\xi)\jp{\xi}^{2s}.$$
By the Leibniz formula, $B\in\mathscr S^{2m+2s}(G)$,
and $B(x,\xi)\geq 0$ due to the definition of $M$.
Consequently, by Theorem \ref{THM:Garding}, we have
$$\Repa(\Op(B)u,u)_{L^2(G)}\geq -C\|u\|_{H^{m+s-1/2}(G)}.$$
Recall that for the bi-invariant Laplace-Beltrami operator $\Lap_G$ on $G$,
the symbol of $I-\Lap_G$ is $\jp{\xi}^2$, so that
$\|Au\|_{H^s(G)}^2=(A^*(I-\Lap_G)^{s}Au,u)_{L^2(G)}$.
On the other hand,
$$\Op(B)+A^*(I-\Lap_G)^{s}A-M^2(I-\Lap_G)^{m+s}\in\Psi^{2m+2s-1}(G)$$
because its symbol is in $\mathscr S^{2m+2s-1}(G)$ by
the composition formula for
pseudo-differen\-tial operators 
(\cite[Thm~10.7.9]{RT-book}) combined with 
the formula for the adjoint operator
(\cite[Thm~10.7.10]{RT-book}).
Combining these facts we obtain the statement of
Corollary~\ref{COR:Ga-Kgothm} by Theorem
\ref{THM:su2-Sobolev}.

\end{document}